\documentclass{amsart} 

\usepackage{lmodern} 
\usepackage{microtype}
\usepackage[english]{babel} 
\usepackage{mathrsfs} 
\usepackage{mathtools}

\theoremstyle{plain} 
\newtheorem{theorem}{Theorem} 
\newtheorem{lemma}[theorem]{Lemma}

\theoremstyle{definition} 
\newtheorem{problem}{Problem}

\theoremstyle{remark} 
\newtheorem*{claim}{Claim}

\DeclareMathOperator{\Arg}{Arg} 
\DeclareMathOperator{\dist}{dist} 
\DeclareMathOperator*{\esssup}{ess\,sup} 
\newcommand{\pv}{\mathrm{p.v.}}

\begin{document} 
\title{Maximal norm Hankel operators} 
\date{\today} 

\author{Ole Fredrik Brevig} 
\address{Department of Mathematics, University of Oslo, 0851 Oslo, Norway} 
\email{obrevig@math.uio.no}

\author{Kristian Seip} 
\address{Department of Mathematical Sciences, Norwegian University of Science and Technology (NTNU), 7491 Trondheim, Norway} 
\email{kristian.seip@ntnu.no}

\begin{abstract}
	A Hankel operator $\mathbf{H}_\varphi$ on the Hardy space $H^2$ of the unit circle with analytic symbol $\varphi$ has minimal norm if $\|\mathbf{H}_\varphi\|=\|\varphi \|_2$ and maximal norm if $\|\mathbf{H}_\varphi\| = \|\varphi\|_\infty$. The Hankel operator $\mathbf{H}_\varphi$ has both minimal and maximal norm if and only if $|\varphi|$ is constant almost everywhere on the unit circle or, equivalently, if and only if $\varphi$ is a constant multiple of an inner function. We show that if $\mathbf{H}_\varphi$ is norm-attaining and has maximal norm, then $\mathbf{H}_\varphi$ has minimal norm. If $|\varphi|$ is continuous but not constant, then $\mathbf{H}_\varphi$ has maximal norm if and only if the set at which $|\varphi|=\|\varphi\|_{\infty}$ has nonempty intersection with the spectrum of the inner factor of $\varphi$. We obtain further results illustrating that the case of maximal norm is in general related to ``irregular'' behavior of $\log |\varphi|$ or the argument of $\varphi$ near a ``maximum point'' of $|\varphi|$. The role of certain positive functions coined apical Helson--Szeg\H{o} weights is discussed in the former context. 
\end{abstract}


\thanks{The research of Seip is supported by Grant 275113 of the Research Council of Norway.}

\maketitle

\section{Introduction} This paper along with its antecedent \cite{Brevig22} grew out of a desire to understand the cases of equality in the most basic norm estimates for Hankel operators in the classical setting of the unit circle $\mathbb{T}$ of the complex plane. We equip as usual $\mathbb{T}$ with normalized Lebesgue arc length measure and define the Hardy space $H^p$ for $1\leq p\leq \infty$ as the subspace of $L^p = L^p(\mathbb{T})$ consisting of the functions whose Fourier coefficients are supported on $\{0,1,2,\ldots\}$. To define Hankel operators, we let $\overline{H^2}$ denote the subspace of $L^2$ consisting of the complex conjugates of functions in $H^2$, and we let $P$ and $\overline{P}$ stand for the orthogonal projections from $L^2$ to respectively $H^2$ and $\overline{H^2}$. Every \emph{symbol} $\varphi$ in $H^2$ induces a (densely defined) Hankel operator 
\begin{equation}\label{eq:hankel} 
	\mathbf{H}_\varphi f \coloneqq \overline{P}(\overline{\varphi} f) 
\end{equation}
from $H^2$ to $\overline{H^2}$. We have plainly 
\begin{equation}\label{eq:hankelnorm} 
	\|\varphi\|_2 \leq \|\mathbf{H}_\varphi\| \leq \|\varphi\|_\infty, 
\end{equation}
and our problem is to identify, in terms of function theoretic properties of $\varphi$, the cases of equality in either of these norm estimates. The Hankel operator $\mathbf{H}_\varphi$ is said to have \emph{minimal norm} if the lower bound in \eqref{eq:hankelnorm} is attained and \emph{maximal norm} if the upper bound in \eqref{eq:hankelnorm} is attained. 

A function $I$ in $H^2$ is \emph{inner} if $|I|=1$ almost everywhere on $\mathbb{T}$. It is plain that $\mathbf{H}_\varphi$ has both minimal and maximal norm if and only if $\|\varphi\|_2 = \|\varphi\|_\infty$. The latter equality is equivalent to the assertion that $|\varphi|$ is constant almost everywhere on $\mathbb{T}$ or, equivalently, that $\varphi = CI$ for a constant $C$ and an inner function $I$. Minimal norm Hankel operators were recently investigated by the first-named author who established the following result. 
\begin{theorem}[{\cite[Thm.~1]{Brevig22}}]
	If $\mathbf{H}_\varphi$ has minimal norm, then $\mathbf{H}_\varphi$ has maximal norm. 
\end{theorem}

We will say that the Hankel operator $\mathbf{H}_\varphi$ is \emph{norm-attaining} if there is a function $g$ in the unit sphere of $H^2$ such that $\|\mathbf{H}_\varphi\|=\|\mathbf{H}_\varphi g\|_2$. To place our results in context, we begin with the following observation. The proof is similar to that of a result of Adamjan, Arov, and Kre\u{\i}n (see \cite{AAK68} or e.g. \cite[Thm.~1.4]{Peller03}).
\begin{theorem}\label{thm:NA} 
	If $\mathbf{H}_\varphi$ is norm-attaining and has maximal norm, then $\mathbf{H}_\varphi$ has minimal norm. 
\end{theorem}

Since compact operators are norm-attaining, it follows from Hartman's theorem~\cite{Hartman58} and a result of Sarason \cite[Thm.~1~(iv)]{Sarason75} that if the symbol $\varphi$ has vanishing mean oscillation, then $\mathbf{H}_\varphi$ is norm-attaining. In particular, if $\varphi$ is continuous on $\mathbb{T}$ and $\mathbf{H}_\varphi$ has maximal norm, then $\varphi = CB$ for a constant $C$ and a finite Blaschke product $B$. If $\varphi = CI$ for a constant $C$ and an inner function $I$, then $\mathbf{H}_\varphi$ is plainly norm-attaining with $g = I$. We refer to \cite{Yoshino02} for a description of the symbols generating norm-attaining Hankel operators. 

In view of this preliminary discussion, it remains to identify those maximal norm Hankel operators that are not minimal norm Hankel operators. The main purpose of this paper is to initiate the study of such operators.

We begin by recasting our problem in function-theoretic terms via Nehari's theorem \cite{Nehari57}, which states that 
\begin{equation}\label{eq:nehari} 
	\|\mathbf{H}_\varphi\| = \inf_{\psi \in L^\infty}\big\{\|\psi\|_\infty \,:\, P\psi = \varphi \big\} 
\end{equation}
for every symbol $\varphi$ in $H^2$. Let $H^\infty_0$ denote the subspace of $H^\infty$ of functions vanishing at the origin. It follows from \eqref{eq:nehari} that $\mathbf{H}_\varphi$ has maximal norm if and only if 
\begin{equation}\label{eq:badapp} 
	\|\varphi\|_\infty = \inf_{f \in H^\infty_0} \|\varphi - \overline{f}\|_\infty. 
\end{equation}
We will say that $\varphi$ is \emph{saturated} if \eqref{eq:badapp} holds. Some terminology is required to state our results. The \emph{super-level sets} and \emph{sub-level sets} of a bounded function $\varphi$ are defined, respectively, as 
\begin{align*}
	L_\varphi^+(\delta) \coloneqq \left\{e^{i\theta} \in \mathbb{T} \,:\, |\varphi(e^{i\theta})| \geq \|\varphi\|_\infty \delta \right\}, \\
	L_\varphi^-(\delta) \coloneqq \left\{e^{i\theta} \in \mathbb{T} \,:\, |\varphi(e^{i\theta})| < \|\varphi\|_\infty \delta \right\}, 
\end{align*}
where $0 \leq \delta \leq 1$. As usual, we factor $\varphi$ as $\varphi = \Phi I$, where $I$ is an inner function and $\Phi$ is an outer function determined by $|\varphi|$ on $\mathbb{T}$ (see \eqref{eq:outer} below). Note that $L_\varphi^\pm = L_\Phi^\pm$, since $|I|=1$ almost everywhere on $\mathbb{T}$. Recall that an inner function $I$ may be written as $I = BS$, where $B$ is a Blaschke product and $S$ is a singular inner function. The \emph{spectrum} of $I$, denoted by $\sigma(I)$ in what follows, is the subset of $\mathbb{T}$ defined as the union of the accumulation points of the zeros of $B$ and the support of the singular measure associated to $S$.

So far, we have seen that functions in $H^\infty$ of constant modulus are always saturated and that functions in $H^\infty$ that have vanishing mean oscillation on $\mathbb{T}$, are saturated if and only if they have constant modulus. By a theorem of Sarason \cite[Thm.~3]{Sarason75}, a simple criterion for $\varphi$ to have vanishing mean oscillation is that $|\varphi|$ be continuous on the closed unit disc. Note in particular that in this case $\sigma(I)=\emptyset$. Our second result is the following complete description of saturated functions in $H^\infty$ that have continuous but not constant modulus on $\mathbb{T}$. 
\begin{theorem}\label{thm:cont} 
	Suppose $\varphi=\Phi I$ is a function in $H^\infty$ such that $|\varphi|$ is continuous but not constant on $\mathbb{T}$. Then $\varphi$ is saturated if and only if $L_\Phi^+(1)\cap \sigma(I) \neq \emptyset$. 
\end{theorem}

It is not difficult to construct examples of saturated functions in $H^\infty$ that are not of the form $\varphi = CI$ from Theorem~\ref{thm:cont}. One of the simplest possibilities is $\varphi = \Phi I$, for
\[\Phi(z) = 1+z \qquad \text{and} \qquad I(z) = \exp\left(\frac{z+1}{z-1}\right).\]
Here it is plain that $L_\Phi^+(1)\cap \sigma(I) = \{1\}$, so it follows from Theorem~\ref{thm:cont} that $\|H_\varphi\| = \|\varphi\|_\infty = \|\Phi\|_\infty = 2$. 

Theorem~\ref{thm:cont} is a consequence of two general results, the first of which reads as follows. 
\begin{theorem}\label{thm:badapp} 
	Suppose that $\varphi = \Phi I$ is in $H^\infty$. If for every $0 < \delta < 1$ the interior of $L_\Phi^+(\delta)$ has nonempty intersection with $\sigma(I)$, then $\varphi$ is saturated. 
\end{theorem}

Before stating our next result, which is a partial converse to Theorem~\ref{thm:badapp}, we recall that a positive function $w$ on $\mathbb{T}$ is a Helson--Szeg\H{o} weight if 
\begin{equation}\label{eq:HS} 
	w=\exp(u+\tilde{v}), 
\end{equation}
where $u$ and $v$ are in $L^\infty$ and $\|v\|_\infty < \pi/2$. In \eqref{eq:HS} and in what follows $\tilde{v}$ denotes the conjugate function of $v$. In our problem, we only need the bound on $v$ at points where $|\varphi|$ is close to $\| \varphi \|_\infty$. We therefore make the following definition: A bounded function $w$ is an \emph{apical Helson--Szeg\H{o} weight} if there exists some $0<\delta<1$ such that 
\begin{equation}\label{eq:aHS} 
	w=\exp(u+v), 
\end{equation}
where $u$ is in $L^{\infty}$, $v$ is in $L^1$, and 
\begin{equation}\label{eq:aHSsup} 
	\esssup_{L_w^+(\delta)} |\tilde{v}|<\frac{\pi}{2}. 
\end{equation}
Note that we have changed $\tilde{v}$ to $v$ when going from \eqref{eq:HS} and \eqref{eq:aHS}, because the global condition on $w$ is now just that $\log{w}$ be integrable. Clearly, an ordinary bounded Helson--Szeg\H{o} weight is also an apical Helson--Szeg\H{o} weight, but the converse is far from true. When $w$ is an apical Helson--Szeg\H{o} weight, $\log{w}$ may for example fail to be of bounded mean variation. When we wish to specify for which $0<\delta<1$ the estimate \eqref{eq:aHSsup} holds, we will say that $w$ is an \emph{apical Helson--Szeg\H{o} weight at level $\delta$}. 
\begin{theorem}\label{thm:HS} 
	Suppose that $\varphi = \Phi I$ is in $H^\infty$ and that there exists some $0<\delta<1$ such that 
	\begin{enumerate}
		\item[(i)] $|\varphi|$ is an apical Helson--Szeg\H{o} weight at level $\delta$; 
		\item[(ii)] the interior of $L_\Phi^-(\delta)$ is nonempty and includes $\sigma(I)$. 
	\end{enumerate}
	Then $\varphi$ is not saturated. 
\end{theorem}

Condition (i) of Theorem~\ref{thm:HS} requires $|\varphi|$ to be, in an essential way, bounded away from $0$ close to the points at which it is near its maximum. The precise degree of nonvanishing, however, may be rather less severe than that of an ordinary Helson--Szeg\H{o} weight, as we will exemplify below. 

Our final theorem exhibits a type of vanishing near a maximum point of $|\varphi|$ that does indeed guarantee that $\varphi$ be saturated. In this result, we use the notation $\Gamma(\vartheta)$ for the arc $\Gamma(\vartheta) \coloneqq \{e^{i\theta}\,:\, |\theta|\leq \vartheta\}$ with $0<\vartheta\leq \pi$.
\begin{theorem}\label{thm:outer} 
	Suppose that $\varphi$ is a function in $H^\infty$ such that 
	\begin{equation}\label{eq:limit} 
		\lim_{\vartheta \to 0^+} \esssup_{\Gamma(\vartheta) \cap L_\varphi^-(1)} |\varphi| = 0 
	\end{equation}
	and that there exists a sequence $(\vartheta_n)_{n\geq1}$ with $\vartheta_n \to 0^+$ such that 
	\begin{equation}\label{eq:asymp} 
		|\Gamma(\vartheta_n) \cap L_\varphi^-(1)| \geq a |\Gamma(\vartheta_n)| \qquad \text{and} \qquad |\Gamma(\vartheta_n) \cap L_\varphi^+(1)| \geq b |\Gamma(\vartheta_n)| 
	\end{equation}
	for positive constants $a$ and $b$. Then $\varphi$ is saturated. 
\end{theorem}

One may think of Theorem~\ref{thm:outer} as a counterpart to Theorem~\ref{thm:badapp}, with the discontinuity in the argument (represented by a point in $\sigma(I)$) replaced by a particular kind of discontinuity in $\log|\varphi|$. Unfortunately, however, this counterpart is only of a rudimentary nature, and we are for example far from having a full description of saturated outer functions that are discontinuous at only one point. 

The next section contains the proofs of the results stated above, while the third and final section of the paper presents a brief discussion of possible next steps in the investigation of saturated functions in $H^{\infty}$.

\section{Proofs} 

In preparation for the proof of Theorem~\ref{thm:NA}, we record the following consequence of the Cauchy--Schwarz inequality and the upper bound in \eqref{eq:hankelnorm}. If $\varphi$ is a nontrivial function in $H^\infty$ and $g$ is in the closed unit ball of $H^2$, then the function
\[f \coloneqq \|\varphi\|_\infty^{-1} \, g \, \overline{\mathbf{H}_\varphi g}\]
is in the closed unit ball of $H^1$. 
\begin{proof}
	[Proof of Theorem~\ref{thm:NA}] We may assume without loss of generality that $\|\varphi\|_\infty \neq 0$. If $\mathbf{H}_\varphi$ is norm-attaining and has maximal norm, then there is some $g$ in the unit sphere of $H^2$ such that 
	\begin{equation}\label{eq:NA1} 
		\|\varphi\|_\infty^2 = \|\mathbf{H}_\varphi\|^2 = \|\mathbf{H}_\varphi g\|_2^2 = \big\langle \mathbf{H}_\varphi g, \mathbf{H}_\varphi g \big\rangle = \big\langle \overline{\varphi} g, \mathbf{H}_\varphi g \big\rangle = \big\langle g \overline{\mathbf{H}_\varphi g} , \varphi \big\rangle, 
	\end{equation}
	where we used the definition of $\mathbf{H}_\varphi$ from \eqref{eq:hankel}. Since $g$ is a nontrivial function in $H^2$ and $\mathbf{H}_\varphi g$ is a nontrivial function in $\overline{H^2}$, we find that $f = \|\varphi\|_\infty^{-1} g \overline{\mathbf{H}_\varphi g}$ is a nontrivial function in the closed unit ball of $H^1$. We deduce from this and \eqref{eq:NA1} that
	\[\|\varphi\|_\infty = \langle f, \varphi \rangle \leq \|f\|_1 \|\varphi\|_\infty \leq \|\varphi\|_\infty,\]
	which implies that $\langle f, \varphi \rangle = \|f\|_1 \|\varphi\|_\infty$. Since $f$ is a nontrivial function in $H^1$, it can only vanish on a subset of $\mathbb{T}$ of measure $0$. This means that $\varphi = \|\varphi\|_\infty f |f|^{-1}$ almost everywhere on $\mathbb{T}$ and consequently that $\varphi = C I$ for a constant $C\neq0$ and an inner function $I$. It follows that $\|\varphi\|_2 = \|\varphi\|_\infty$, and hence $\mathbf{H}_\varphi$ has minimal norm. 
\end{proof}

As Theorem~\ref{thm:cont} relies on both Theorem~\ref{thm:badapp} and Theorem~\ref{thm:HS}, we will establish the latter two results first. Before we proceed with the proof of Theorem~\ref{thm:badapp}, we will recall a few results (which can be found e.g.~in \cite[Sec.~II.6]{Garnett07}) on analytic continuation of functions in Hardy spaces across $\mathbb{T}$. Suppose that $f$ is a function in $H^1$ and that $\Gamma$ is an open arc on $\mathbb{T}$. If $f$ is analytic across $\Gamma$, then both its inner factor $I$ and its outer factor $F$ are analytic across $\Gamma$. Moreover, an inner function $I$ is analytic across each open arc $\Gamma$ on $\mathbb{T}$ which does not intersect the spectrum $\sigma(I)$.
\begin{proof}
	[Proof of Theorem~\ref{thm:badapp}] We will argue by contradiction and assume that there is some $f$ in $H^\infty_0$ such that $\|\varphi-\overline{f}\|_\infty < \|\varphi\|_\infty$. The left-hand side of this inequality is plainly nonzero, so there is some $0<\varepsilon<1$ such that 
	\begin{equation}\label{eq:delta} 
		\|\varphi-\overline{f}\|_\infty = \varepsilon\|\varphi\|_\infty. 
	\end{equation}
	We will choose $\delta\coloneqq (1+\varepsilon)/2$. Since the interior of $L_\Phi^+(\delta)$ is assumed to have nonempty intersection with $\sigma(I)$, there is an open arc $\Gamma$ in $L_\Phi^+(\delta)$ that has nonempty intersection with $\sigma(I)$. We know that $|\varphi| \geq \delta \|\varphi\|_\infty$ on $\Gamma$ by definition, so it follows from \eqref{eq:delta} that $|f| \geq (1-\varepsilon)/2 \|\varphi\|_\infty $ almost everywhere on $\Gamma$. It is also plain that $\|f\|_\infty \leq (1+\varepsilon)\|\varphi\|_\infty$. Expanding $|\varphi-\overline{f}|^2$ and using these estimates, we deduce from \eqref{eq:delta} that 
	\begin{equation}\label{eq:Arg} 
		|\Arg{\varphi f}| \leq \arccos{\frac{1-\varepsilon}{4}} < \frac{\pi}{2} 
	\end{equation}
	almost everywhere on $\Gamma$. Let $\nu$ be the harmonic function in the unit disc defined by the boundary values
	\[\nu \coloneqq 
	\begin{cases}
		-\Arg{\varphi f}, & \text{on } \Gamma; \\
		0, & \text{on } \mathbb{T} \setminus \Gamma. 
	\end{cases}\]
	The function $g \coloneqq \exp(-\tilde{\nu}+i \nu)$ is analytic in the unit disc and maps the unit disc to a cone in the right-half plane due to \eqref{eq:Arg}. Consequently, $g$ is in $H^1$. By construction, the $H^1$ function $h \coloneqq \varphi f g$ is positive on $\Gamma$, whence it is analytic across this arc (see e.g.~\cite[Lem.~IV.1.10]{Garnett07}). In particular, the inner factor of $h$ is analytic across $\Gamma$. This means that $I$ (the inner factor of $\varphi$) is also analytic across $\Gamma$, since $\sigma(I)$ is a subset of the spectrum of the inner factor of $h$. We have arrived at a contradiction since $\sigma(I)$ has nonempty intersection with $\Gamma$. 
\end{proof}

Recall (from e.g.~\cite[Sec.~II.5]{Garnett07}) that the outer factor of $\varphi = \Phi I$ can be recovered from the modulus of $\varphi$ on $\mathbb{T}$ by the formula 
\begin{equation}\label{eq:outer} 
	\Phi(z) = \exp\left(\int_{-\pi}^\pi \frac{e^{i\theta}+z}{e^{i\theta}-z}\log|\varphi(e^{i\theta})|\,\frac{d\theta}{2\pi}\right). 
\end{equation}
In particular, if $|\Phi|=|\varphi|$ is a Helson--Szeg\H{o} weight, then it follows from \eqref{eq:HS} and \eqref{eq:outer} that $\Phi = \exp(u + i \tilde{u}+\tilde{v}-iv)$, where $u$ and $v$ are in $L^\infty$ with $\|v\|_\infty < \pi/2$. Similarly, if $|\Phi|=|\varphi|$ is an apical Helson--Szeg\H{o} weight at level $\delta$, then it follows from \eqref{eq:aHS} and \eqref{eq:outer} that $\Phi = \exp(u+i \tilde{u}+v+i\tilde{v})$, where $u$ is in $L^\infty$ and $\tilde{v}$ satisfies \eqref{eq:aHSsup}.
\begin{proof}
	[Proof of Theorem~\ref{thm:HS}] Our task is to construct a function $f$ in $H^\infty_0$ such that 
	\begin{equation}\label{eq:gapp} 
		\|\varphi-\overline{f}\|_\infty < \|\varphi\|_\infty. 
	\end{equation}
	To satisfy the requirement that $f$ be in $H^\infty_0$, we consider functions of the form $f(z) = z g(z)$ for $g$ in $H^\infty$. The function $g$ will consist of two factors arising from, respectively, the inner and outer factors of $\varphi = \Phi I$. We begin with $I$. By assumption, there is some $0<\delta<1$ such that $\sigma(I)$ is a subset of the interior of $L_\Phi^-(\delta)$. Being a nonempty open set, the latter set can be written as a union of open arcs. By compactness of $\sigma(I)$, we may pick a finite collection of such arcs that covers $\sigma(I)$. Letting $\Gamma$ denote the union of these arcs, we may define the argument of $h(z) \coloneqq z I(z) $ as a $C^\infty$ function on $\mathbb{T} \setminus \Gamma$. Consequently, since $\Gamma$ is a finite union of open arcs, we may find a $C^{\infty}$ function $\nu$ on $\mathbb{T}$ such that
	\[\nu = -\arg{h}\]
	on $\mathbb{T} \setminus \Gamma$. Since $\nu$ is smooth, the function $g_1 \coloneqq \exp(-\tilde{\nu}+i\nu)$ is in $H^\infty$ and $|g_1|$ is bounded below on $\mathbb{T}$. For the outer part of $\varphi$, we use the assumption that $|\Phi|$ be an apical Helson--Szeg\H{o} weight at level $\delta$ and choose $g_2 \coloneqq \exp(-u-i\tilde{u})$. Since $u$ is in $L^\infty$, it follows that $g_2$ is in $H^\infty$ and that $|g_2|$ is bounded below on $\mathbb{T}$. Set
	\[f \coloneqq \varepsilon z g_1 g_2,\]
	for some $\varepsilon>0$. If $\varepsilon \leq (1-\delta) \|\varphi\|_\infty /(2\|g_1\|_\infty \|g_2\|_\infty)$, then $|\varphi-\overline{f}| \leq (1+\delta)/2\|\varphi\|_\infty$ on $L_\Phi^-(\delta)$. Conversely, since $L_\Phi^+(\delta)$ is contained in $\mathbb{T}\setminus \Gamma$ we have
	\[\big|\varphi-\overline{f}\big| = \big||\varphi|-\varepsilon e^{-i\tilde{v}} |g_1| |g_2|\big|\]
	on $L_\Phi^+(\delta)$ by construction. Using \eqref{eq:aHSsup} and the fact that both $|g_1|$ and $|g_2|$ are bounded below on $\mathbb{T}$, it follows that there is some sufficiently small $\varepsilon>0$ such that $|\varphi-\overline{f}| < \|\varphi\|_\infty$ almost everywhere on $L_\Phi^+(\delta)$. Combining the estimates on $L_\Phi^-(\delta)$ and on $L_\Phi^+(\delta)$, we see that \eqref{eq:gapp} holds. 
\end{proof}

The following sufficient condition for being an apical Helson--Szeg\H{o} weight may be of some independent interest. We will primarily apply it in the proof of Theorem~\ref{thm:cont}.
\begin{lemma}\label{lem:HS} 
	Suppose $w$ is a bounded positive function on $\mathbb{T}$ such that $\log{w}$ is integrable and fix $0<\delta<1$. If there is an $\varepsilon$ in $(0,\delta)$ such that 
	\begin{equation}\label{eq:dist} 
		\dist\big(L_w^+(\delta),L_w^-(\varepsilon)\big)>0, 
	\end{equation}
	then $w$ is an apical Helson--Szeg\H{o} weight at level $\delta$. 
\end{lemma}
\begin{proof}
	We may assume without loss of generality that $\|w\|_\infty=1$. We set
	\[u_R \coloneqq \max(\log{w},-R) \qquad \text{and} \qquad v_R \coloneqq R + \min(\log{w},-R)\]
	for $R>0$. Note that $u_R$ is in $L^\infty$ for each fixed $R$. Since $\log{w}$ is integrable, we have plainly that $v_R$ is in $L^1$ for each fixed $R$ and that $\|v_R\|_1 \to 0 $ when $R\to\infty$. Let $E_R \coloneqq \{\theta \in [-\pi,\pi] \,:\, e^{i\theta} \in L_w^-(e^{-R})\}$. We see from the definition of $v_R$ that
	\[\tilde{v}_R(e^{it}) = \pv \int_{E_R} \cot\left(\frac{t-\theta}{2}\right)\,v_R(e^{i\theta})\,\frac{d\theta}{2\pi}.\]
	We pick $R$ so large that $e^{-R} \leq \varepsilon$. It now follows that there is a constant $C>0$, which only depends on the distance in \eqref{eq:dist}, such that $|\tilde{v}_R| \leq C \|v_R\|_1$ on $L_w^+(\delta)$. Writing $w=\exp(u_R+v_R)$ for $R$ sufficiently large, we see that $w$ is an apical Helson--Szeg\H{o} weight at level $\delta$. 
\end{proof}
\begin{proof}
	[Proof of Theorem~\ref{thm:cont}] We first note that if $L_\Phi^+(1) \cap \sigma(I) \neq \emptyset$, then Theorem~\ref{thm:badapp} applies so that $\varphi$ is saturated. On the other hand, if $L_\Phi^+(1) \cap \sigma(I) = \emptyset$, then we use that $|\Phi|=|\varphi|$ is continuous and that $\sigma(I)$ is closed to infer that there exists a $0<\delta<1$ such that we also have $L_\Phi^+(\delta) \cap \sigma(I) = \emptyset$. Since $|\Phi|$ is assumed to be nonconstant, we may also assume that $L_\Phi^-(\delta)$ is nonempty. By the continuity of $|\Phi|$, we find that $L_\Phi^-(\delta)$ is itself an open set that includes $\sigma(I)$. In order to apply Theorem~\ref{thm:HS}, it therefore remains to check that $|\Phi|$ is an apical Helson--Szeg\H{o} weight at level $\delta$. However, this follows at once from the continuity of $|\Phi|$ and Lemma~\ref{lem:HS}. 
\end{proof}

We require two preliminary results for the proof of Theorem~\ref{thm:outer}. The first is the following direct consequence of \eqref{eq:badapp}.
\begin{lemma}\label{lem:innerout} 
	If $\varphi$ is saturated, then so is $I \varphi$ for any inner function $I$. 
\end{lemma}
\begin{proof}
	Suppose that $I\varphi$ is not saturated. Then there is a function $f$ in $H^\infty_0$ such that $\|I\varphi - \overline{f}\|_\infty < \|I\varphi\|_\infty$. However, this is equivalent to $\|\varphi- \overline{If}\|_\infty < \|\varphi\|_\infty$, so $\varphi$ is not saturated. 
\end{proof}

The second preliminary result follows from Muckenhoupt's theorem \cite{M72}.
\begin{lemma}\label{lem:muckenhoupt} 
	Let $w$ be a Helson--Szeg\H{o} weight. Then there exists a positive constant $C$ and $1<p<2$ such that
	\[\int_\Gamma w \leq C \frac{|\Gamma|^p}{|E|^p} \int_E w,\]
	for every arc $\Gamma$ and every measurable set $E\subseteq \Gamma$. 
\end{lemma}
\begin{proof}
	Since $w$ is a Helson--Szeg\H{o} weight, or equivalently an $(A_2)$ weight, we know that $w$ satisfies the $(A_p)$ condition for some $1<p<2$ (see e.g.~\cite[Cor.~VI.6.10 (b)]{Garnett07}). The stated result then follows by an application of Muckenhoupt's theorem (see \cite{M72} or e.g.~\cite[Thm.~VI.6.1]{Garnett07}) to the function $\chi_E$. 
\end{proof}

We will appeal to Lemma~\ref{lem:muckenhoupt} several times in the proof of Theorem~\ref{thm:outer}. However, it is only in its final application that it matters that we use the $(A_p)$ condition for some $1<p<2$ and not simply the $(A_2)$ condition. In the proof of Theorem~\ref{thm:outer} we will locally adopt the notation $f \ll g$, which means that there is a positive constant $B$ such that $|f(x)| \leq B g(x)$ for all relevant $x$. The constant $B$ will be allowed to depend on $\varphi$.
\begin{proof}
	[Proof of Theorem~\ref{thm:outer}] By Lemma~\ref{lem:innerout}, it is sufficient to consider the case that $\varphi$ is outer. We argue by contradiction and assume that there is some $f$ in $H^\infty_0$ and some $0<\varepsilon<1$ such that
	\[\|\varphi-\overline{f}\|_\infty = \varepsilon \|\varphi\|_\infty.\]
	As in the proof of Theorem~\ref{thm:badapp}, it follows that
	\[|f| \geq (1-\varepsilon)\|\varphi\|_\infty \qquad \text{and} \qquad |\Arg(\varphi f)| \leq \arccos\left(\frac{1-\varepsilon}{1+\varepsilon}\right)<\frac{\pi}{2}\]
	hold almost everywhere on $L_\varphi^+(1)$. Moreover, if
	\[\nu \coloneqq
	\begin{cases}
		-\Arg(\varphi f), & \text{on } L_\varphi^+(1); \\
		0, & \text{on } L_\varphi^-(1), 
	\end{cases}\]
	then $g \coloneqq \exp(-\widetilde{\nu}+i\nu)$ is in $H^1$. Clearly, $|g|$ is a Helson--Szeg\H{o} weight by \eqref{eq:HS}, and the $H^1$ function $h \coloneqq \varphi f g$ is positive on $L_\varphi^+(1)$. Since $f$ is in $H^\infty_0$, we know that $h(0)=0$. We may therefore write $h = u+iv$ with $u = -\tilde{v}$. 
	
	We next set
	\[E_n := \Gamma(\vartheta_n) \cap L_\varphi^{-1}(1) \qquad \text{and} \qquad F_n \coloneqq\Gamma(\vartheta_n) \cap L_\varphi^+(1)\]
	and claim that 
	\begin{equation}\label{eq:EnFn} 
		\int_{E_n} |u| \ll \esssup_{E_n} |\varphi| \int_{F_n} |u|. 
	\end{equation}
	To verify this, we first use the trivial estimate $\|f\|_\infty \leq (1+\varepsilon) \|\varphi\|_\infty \leq 2 \|\varphi\|_\infty$ to see that 
	\begin{equation}\label{eq:ug} 
		|u| \leq |h| = |\varphi f g| \leq \left(\esssup_{E_n}{|\varphi|}\right) 2 \|\varphi\|_\infty |g| 
	\end{equation}
	on $E_n$. We next use that $E_n$ is a subset of $\Gamma(\vartheta_n)$ and then apply Lemma~\ref{lem:muckenhoupt} with respect to the subset $F_n$ of $\Gamma(\vartheta_n)$ to infer that 
	\begin{equation}\label{eq:Eng} 
		\int_{E_n} |g| \leq \int_{\Gamma(\vartheta_n)} |g| \ll \int_{F_n} |g| \leq \frac{1}{(1-\varepsilon)\|\varphi\|_\infty^2} \int_{F_n} |u|. 
	\end{equation}
	Here we used the assumption that $|F_n| \geq b |\Gamma(\vartheta_n)|$ from \eqref{eq:asymp} and, in the final estimate, that $|h| \geq (1-\varepsilon)\|\varphi\|_\infty^2 |g|$ almost everywhere on $L_\varphi^+(1)$ and that $h=u$ on $L_\varphi^+(1)$. Now \eqref{eq:EnFn} follows from \eqref{eq:ug} and \eqref{eq:Eng}.
	
	We next claim that if $n$ is sufficiently large, then 
	\begin{equation}\label{eq:FnEn} 
		\int_{F_n} |u| \ll \int_{E_n} |u|. 
	\end{equation}
	If we can establish \eqref{eq:FnEn}, then we would be done since \eqref{eq:EnFn} and \eqref{eq:FnEn} together contradict the assumption \eqref{eq:limit}. (Note that the left-hand side of \eqref{eq:FnEn} is nonzero for each fixed $n\geq1$ since $u$ is positive on $F_n$ and $|F_n|>0$.)
	
	To prove \eqref{eq:FnEn}, we start by writing $u = u_0 + u_1$, where $u_0$ is the conjugate function of $-v \chi_{\Gamma(2\vartheta_n)}$. By the weak-type $L^1$ bound for the conjugate function, we find for $\lambda>0$ that 
	\begin{equation}\label{eq:weakL1} 
		|\{|u_0|> \lambda \}| \ll \frac{1}{\lambda} \|v \chi_{\Gamma(2\vartheta_n)}\|_1 \leq \frac{1}{\lambda} \|h \chi_{\Gamma(2\vartheta_n) \cap L_\varphi^{-1}(1)}\|_1, 
	\end{equation}
	where we first used that $v\equiv 0$ on $L_\varphi^+(1)$ and then that $|v| \leq |h|$. By the same argument as used to establish \eqref{eq:EnFn} above, we obtain next that 
	\begin{equation}\label{eq:weakL13} 
		\|h \chi_{\Gamma(2\vartheta_n) \cap L_\varphi^{-1}(1)}\|_1 \ll \left(\esssup_{\Gamma(2\vartheta_n) \cap L_\varphi^-(1)}{|\varphi|}\right) \int_{F_n} |u|. 
	\end{equation}
	By \eqref{eq:limit}, we can choose $n$ so large that the essential supremum in \eqref{eq:weakL13} is as small as we wish. In particular, we will let $n$ be so large that \eqref{eq:weakL1} and \eqref{eq:weakL13} together give 
	\begin{equation}\label{eq:weakL12} 
		|\{|u_0|> \lambda \}| \leq \frac{\min(a,b)}{2\pi \lambda}\int_{F_n} |u|. 
	\end{equation}
	We set $\lambda_n \coloneqq c \vartheta_n^{-1} \int_{F_n} |u|$ for a constant $c>0$ to be chosen later and apply \eqref{eq:weakL12} in combination with the estimate $|F_n| \geq b \vartheta_n/\pi$ from \eqref{eq:asymp} to infer that 
	\begin{equation}\label{eq:Fncapu0} 
		|F_n \cap \{|u_0| \leq \lambda_n \}| \geq b \vartheta_n/(2\pi). 
	\end{equation}
	Using that $|u| \geq (1-\varepsilon) \|\varphi\|_\infty^2 |g|$ almost everywhere on $L_\varphi^+(1)$, then Lemma~\ref{lem:muckenhoupt} and \eqref{eq:Fncapu0}, and finally that $\|u\|_\infty \leq 2\|\varphi\|_\infty^2 |g|$, we obtain 
	\begin{equation}\label{eq:u0small} 
		\int_{F_n \cap \{|u_0| \leq \lambda_n\}} |u| \gg \int_{\Gamma(\vartheta_n)} |g| \geq \int_{F_n} |g| \gg \int_{F_n} |u|. 
	\end{equation}
	Let $C$ denote the total implied constant in the chain of inequalities \eqref{eq:u0small}, which we stress does not depend on our upcoming choice of $c$. By the estimate $u_1 \geq u-|u_0|$, the fact that $u=|u|$ on $L_\varphi^+(1)$, and \eqref{eq:u0small}, we find that
	\[\int_{F_n \cap \{|u_0| \leq \lambda_n\}} u_1 \geq \int_{F_n \cap \{|u_0|\leq \lambda_n\}} \big(|u|-\lambda_n\big) \geq \left(\frac{1}{C}-c\frac{|F_n \cap \{|u_0|\leq \lambda_n\}|}{\vartheta_n}\right)\int_{F_n} |u|.\]
	Since $|F_n| \leq |\Gamma(\vartheta_n)| = \vartheta_n/\pi$, we can choose $c \leq \pi / (2C)$ to ensure that 
	\begin{equation}\label{eq:u1u} 
		\int_{F_n \cap \{|u_0| \leq \lambda_n\}} u_1 \geq\frac{1}{2C} \int_{F_n} |u| 
	\end{equation}
	for all sufficiently large $n$. The next step is to estimate the difference $u_1(e^{it})-u_1(1)$ for $|t|\leq \vartheta_n$. Setting $I_n \coloneqq \{2 \vartheta_n < |\theta| \leq \pi \}$, we have
	\[u_1(e^{i\theta}) = - \pv \int_{I_n} \cot\left(\frac{t-\theta}{2}\right)\, v(e^{i\theta})\,\frac{d\theta}{2\pi}.\]
	We decompose $I_n$ into the dyadic intervals $I_{n,k} \coloneqq\{2^k \vartheta_n \leq |\theta| \leq 2^{k+1}\vartheta_n\}$ for $k=1,2,\ldots,K$ and note that if $|t| \leq \vartheta_n$ and $\theta$ is in $I_{n,k}$, then
	\[\left|\cot\left(\frac{t-\theta}{2}\right)-\cot\left(\frac{-\theta}{2}\right)\right| \ll \frac{1}{\vartheta_n 2^k}.\]
	Consequently,
	\[|u_1(e^{it})-u_1(1)| \ll \sum_{k=1}^{K} \frac{1}{\vartheta_n 2^k} \int_{\Gamma(2^{k+1}\vartheta_n)} |v|.\]
	To estimate this integral, we use first that $v$ is supported on $L_\varphi^-(0)$ and then that $|v| \leq |h|$ to see that
	\[\int_{\Gamma(2^{k+1}\vartheta_n)} |v| \leq 2\|\varphi\|_\infty \left(\esssup_{\Gamma(2^{k+1}\vartheta_n) \cap L_\varphi^-(1)}{|\varphi|}\right) \int_{\Gamma(2^{k+1}\vartheta_n)} |g|.\]
	We appeal to Lemma~\ref{lem:muckenhoupt} and obtain that
	\[\int_{\Gamma(2^{k+1}\vartheta_n)} |g| \ll \left(\frac{|\Gamma(2^{k+1}\vartheta_n)|}{|F_n|}\right)^p \int_{F_n} |g| \leq \left(\frac{2^{k+1}}{b}\right)^p \frac{1}{(1-\varepsilon)\|\varphi\|_\infty^2} \int_{F_n} |u|.\]
	Combining these estimates, we find that
	\[|u_1(e^{it})-u_1(1)| \ll \left(\frac{1}{\vartheta_n} \int_{F_n} |u|\right) \sum_{k=1}^K 2^{k(p-2)} \esssup_{\Gamma(2^{k+1}\vartheta_n) \cap L_\varphi^-(1)}{|\varphi|}.\]
	The sum goes to $0$ as $n\to \infty$, as can be seen by using \eqref{eq:limit} for, say, $2^{k+1} \vartheta_n \leq \sqrt{\vartheta_n}$. We can therefore assume that
	\[|u_1(e^{it})-u_1(1)| \leq \frac{\pi}{8C \vartheta_n}\int_{F_n} |u|\]
	for $|t|\leq \vartheta_n$ and all sufficiently large $n$, where $C$ is the constant in \eqref{eq:u1u}. Inserting this estimate in \eqref{eq:u1u} and using that $|F_n| \leq \vartheta_n\pi$, we find that
	\[u_1(1) \geq \frac{3 \pi }{8 C \vartheta_n} \int_{F_n} |u| \qquad \text{and} \qquad u_1(e^{it}) \geq \frac{\pi }{4 C \vartheta_n} \int_{F_n} |u|\]
	for $|t|\leq \vartheta_n$ and all sufficiently large $n$. We finally choose $c = \pi/(8 C)$ in the definition of $\lambda_n$ to ensure that
	\[\int_{E_n} |u|\geq \int_{E_n \cap \{|u_0|\leq \lambda_n\}} |u| \geq \int_{E_n \cap \{|u_0|\leq \lambda_n\}} \frac{u_1}{2} \geq \frac{a}{16C} \int_{F_n} |u|.\]
	In the final estimate we also used that $|E_n \cap \{|u_0|\leq \lambda_n\}| \geq a \vartheta_n/(2\pi)$ which follows from \eqref{eq:asymp} and \eqref{eq:weakL12}, similarly to how we established \eqref{eq:Fncapu0} above. This completes the proof of \eqref{eq:FnEn}. 
\end{proof}

\section{Concluding remarks} Falling short of giving a full description of the saturated functions in $H^\infty$, we would like to conclude the paper by proposing two problems that may be more manageable.
\begin{problem}
	Describe the outer functions in $H^{\infty}$ that are saturated. 
\end{problem}
By Theorem~\ref{thm:badapp}, an outer function $\varphi$ fails to be saturated if there is a $0<\delta<1$ such that $|\varphi|$ is an apical Helson--Szeg\H{o} weight at level $\delta$ and the interior of $L_{\varphi}^-(\delta)$ is nonempty. It is natural to ask which, if any, of these two conditions may also be necessary for $\varphi$ not to be saturated.

At any rate, it would be desirable to get a better understanding of apical Helson--Szeg\H{o} weights. We will motivate our second problem, which deals with such weights, by presenting the example alluded to below Theorem~\ref{thm:badapp}. To this end, let $U$ be an open subset of $\mathbb{T}$ with $0<|U|<1$. Then $U$ is a finite or countable union of pairwise disjoint open arcs
\[U = \bigcup_{n\geq1} \Gamma_n.\]
For $0<\varepsilon\leq1$, we let $\varepsilon \Gamma_n$ denote the open arc concentric to $\Gamma_n$ with $|\varepsilon \Gamma_n| = \varepsilon |\Gamma_n|$. We will say that $U$ is \emph{thin} if 
\begin{equation}\label{eq:thin} 
	\sup_{e^{i\theta}\in \mathbb{T}\setminus U} \sum_{n\geq1} \frac{|\Gamma_n|}{\dist\big(e^{i\theta},\frac{1}{2}\Gamma_n\big)} < \infty. 
\end{equation}
The set $\bigcup_{n\geq1} \left\{e^{i\theta}\,:\, 2^{-n} < \theta < 2^{-n}+a b^n \right\}$ for $0<a\leq1$ and $0<b\leq 1/2$ is for example thin if and only if $b<1/2$.

Let $e^{i\theta_n}$ be the midpoint of $\Gamma_n$. For a positive number $\kappa$, we define the weight
\[w_\kappa(e^{i\theta}) \coloneqq 
\begin{cases}
	\left(\frac{|\theta-\theta_n|}{\pi |\Gamma_n|}\right)^\kappa, & e^{i\theta} \in \Gamma_n; \\
	1, & e^{i\theta} \in \mathbb{T}\setminus U. 
\end{cases}\]
We make two observations. First, in the most interesting case, when there are infinitely many arcs $J_n$, we have
\[\dist\big(L_w^+(\delta),L_w^-(\varepsilon)\big)=0\]
for all $0<\delta<\varepsilon<1$, so Lemma~\ref{lem:HS} is of no help. Second, it is clear that $w_\kappa$ is a Helson--Szeg\H{o} weight if and only if $0<\kappa<1$. In contrast, however, we will now prove the following.
\begin{claim}
	If $U$ is thin, then $w_\kappa$ is an apical Helson--Szeg\H{o} weight at level $1/2$ for every $\kappa>0$. 
\end{claim}
\begin{proof}[Proof of Claim.] 
	There is a number $\varepsilon_\kappa>0$ such that if $0<\varepsilon\leq\varepsilon_\kappa$, then $\varepsilon \Gamma_n$ is a subset of $L_w^-(1/4)$ for all $n\geq1$. Set $U_\varepsilon \coloneqq \bigcup_{n\geq1} \varepsilon \Gamma_n$,
	\[u_\varepsilon \coloneqq 
	\begin{cases}
		0, & \text{on }\, U_\varepsilon; \\
		\log{w}, & \text{on }\, \mathbb{T}\setminus U_\varepsilon, 
	\end{cases}\]
	and $v_\varepsilon \coloneqq \log{w}-u_\varepsilon$. A computation shows that
	\[\int_{\varepsilon \Gamma_n} |v_{\varepsilon}| \leq C \kappa \varepsilon |\log{\varepsilon}| |\Gamma_n|,\]
	with $C$ independent of $\kappa$ and $\varepsilon$. By the definition of $\tilde{v}_{\varepsilon}$ and the assumption that \eqref{eq:thin} holds, we therefore find that there is a constant $C$ independent of $\kappa$ and $\varepsilon$ such that
	\[|\tilde{v}_{\varepsilon}| \leq C \kappa \varepsilon |\log{\varepsilon}|\]
	on $J_w^-(1/2)$. Hence $w=\exp(u_{\varepsilon}+v_{\varepsilon})$ is the required representation of $w$ for all $\varepsilon$ sufficiently small. 
\end{proof}

The above proof along with that of Lemma~\ref{lem:HS} suggests the following problem.
\begin{problem}
	Is it true that $w$ is an apical Helson--Szeg\H{o} weight if and only if $w^\kappa$ is an apical Helson--Szeg\H{o} weight for all positive numbers $\kappa$? 
\end{problem}
In other words, we ask if \eqref{eq:aHSsup} could be replaced simply by 
\begin{equation}\label{eq:simpleHS} 
	\esssup_{L_w^+(\delta)} |\tilde{v}|<\infty. 
\end{equation}
This would, in particular, imply that $w$ is an apical Helson--Szeg\H{o} weight whenever $\log{w}$ has bounded mean oscillation.

One could object that we may construct more sophisticated examples for which it would be less straightforward to find $v$ such that the left-hand side of \eqref{eq:simpleHS} is arbitrarily small. However, one should also take into account that our construction of $v$ is rather simple-minded in either of the cases treated above. 

\bibliographystyle{amsplain} 
\bibliography{hankelmax}

\end{document}